\documentclass[12pt]{article}

\usepackage{amssymb}

 \usepackage{amsthm}

\newtheorem{thm}{Theorem}
\newtheorem{lem}[thm]{Lemma}
\newtheorem{rmk}{Remark}

\begin{document}

\title{Note on the bondage number of graphs on topological surfaces}

\author{Vladimir Samodivkin \medskip\\ \small Department of Mathematics, UACEG,
 \\ \small Hristo Smirnenski Blv. 1, 1046 Sofia, Bulgaria 
 \\ \small E-mail: vl.samodivkin@uacg.bg}

\renewcommand{\today}{\small August 30, 2012}
\maketitle

\begin{abstract}
The bondage number $b(G)$ of a graph $G$ is the smallest number 
of edges whose removal from $G$ results in a graph with larger 
domination number. 
In this paper we present new upper bounds for $b(G)$ in terms of 
girth, order and Euler characteristic.
\end{abstract}

\smallskip


% \linenumbers

%% main text
\section{Introduction and main results}
\label{parva}
We shall consider graphs without loops and multiple edges. 
 %Multigraphs may have loops or multiple edges. 
An orientable compact 2-manifold $\mathbb{S}_h$ or orientable surface $\mathbb{S}_h$ (see \cite{Ringel}) of genus $h$ is
obtained from the sphere by adding $h$ handles. Correspondingly, a non-orientable compact
2-manifold $\mathbb{N}_k$ or non-orientable surface $\mathbb{N}_k$ of genus $k$ is obtained from the sphere by
adding $k$ crosscaps. The Euler characteristic is defined by
$\chi(\mathbb{S}_h) = 2 - 2h$, $h \geq 0$,  and $\chi(\mathbb{N}_k ) = 2 - k$, $k \geq 1$.
Compact 2-manifolds are called simply surfaces throughout the paper. 
 If a graph $G$ is embedded in a surface $\mathbb{M}$ then the connected components of $\mathbb{M} - G$
are called the faces of $G$. If each face is an open disc then the embedding is called a
2-cell embedding. 
For such a graph $G$, we denote its vertex set,
edge set, face set, maximum degree, and minimum degree by $V (G)$, $E(G)$, $F(G)$, $\Delta(G)$, and
$\delta(G)$, respectively. Set $|G| = |V(G)|$, $\|G\| = |E(G)|$, and $f (G) = |F(G)|$. We call $|G|$ and
$\|G\|$ the order and the size of $G$. 
For a 2-cell embedding in a surface $\mathbb{M}$ the (generalized) Euler's formula states
$|G| - \|G\| + f(G) = \chi(\mathbb{M})$ for any multigraph $G$ that is 2-cell embedded in $\mathbb{M}$ \cite[p. 85]{Ringel}. The Euclidean plane $\mathbb{S}_0$, the projective plane $\mathbb{N}_1$, the torus $\mathbb{S}_1$, and the Klein bottle $\mathbb{N}_2$ are all the surfaces of nonnegative Euler characteristic.
The degree of a face is the length of its boundary walk.
A  face of degree $i$ is called a $i$-face. 
For $i \geq 3$, let $f_i (G)$ denote the number of $i$ -faces in the embedded graph $G$.
We say that two  faces
 are intersecting or adjacent if they share a common vertex or a
common edge, respectively. 
The girth of a graph $G$, denoted as $g(G)$, is the length of a shortest
cycle in $G$. If $G$ has no cycle then $g(G) = \infty$. 

A dominating set for a graph $G$ is a subset $D\subseteq V(G)$ of 
vertices such that every vertex not in $D$ is adjacent to
at least one vertex in $D$. The minimum cardinality of a dominating 
set is called the domination number of $G$.
The concept of domination in graphs has many applications 
in a wide range of areas within the natural and social sciences.
One measure of the stability of the domination number of $G$ 
under edge removal is the bondage number $b(G)$, defined in \cite{FJKR}
(previously called the domination line-stability in \cite{BHNS}) 
as the smallest number of edges whose 
removal from $G$ results in a graph with larger domination number.
In general it is hard to determine the bondage number $b(G)$ 
(see Hu and Xu~\cite{HuXu}), and thus useful to find bounds for it. 

The main result of the paper is the following theorem. 
\begin{thm} \label{edno}
Let $G$ be a graph embeddable on a surface whose Euler
characteristic $\chi$ is as large as possible and let $g(G) = g < \infty$. 
\begin{itemize} 
\item[(i)] Then $b(G) \leq 3 + \frac{8}{g-2} - \frac{4\chi g}{|G|(g-2)}$.
\item[(ii)] If $G$ contains no intersecting $g$-faces, then 
$b(G) \leq 3 + \frac{8g+4}{g^2-g}-4(1+\frac{2}{g-1})\frac{\chi}{|G|}$.
\item[(iii)] If G contains no adjacent $g$-faces, then 
$b(G) \leq \frac{4g(g+1)}{g^2-g-1}(1-\frac{\chi}{|G|}) -1$.
\end{itemize}
\end{thm}
\begin{rmk}
If $G$ is a planar graph with girth $g \geq 4+i$, $i \in \{0,1,2\}$
 then Theorem \ref{edno}(i) leads to $b(G) \leq 6-i$. 
These results were first proved by Fischermann et al.\cite{FRV}.
\end{rmk}
Recently, the following results on bondage number of graphs on surfaces were obtained.
\begin{thm}[Gagarin and Zverovich~\cite{GagarinZverovich}]
\label{thm:GZ}
Let $G$ be a graph embeddable on an orientable surface of 
genus $h$ and a non-orientable surface of genus $k$. Then
$b(G)\leq\min\{\Delta(G)+h+2,\Delta(G)+k+1\}$.
\end{thm}
\begin{thm}[Jia Huang~\cite{Jia Huang}]
\label{thm:H21}
Let $G$ be a graph embeddable on a surface whose Euler
characteristic $\chi$ is as large as possible. If $\chi\leq0$
then $b(G)<\Delta(G)+\sqrt{12-6\chi}+1/2$. If $\chi\leq0$  
then $b(G) \leq \Delta (G) + \frac{\sqrt{8g(2-g)\chi + (3g-2)^2} -g +6}{2(g-2)}$.
\end{thm}
\begin{rmk}
In many cases the bound stated in Theorem \ref{edno}(i)  
is better than those given by Theorems \ref{thm:GZ} and \ref{thm:H21}. 
 Indeed, it is easy to see that if $\chi \leq 0$ then: 
 \begin{itemize}
\item[(a)] $s(\chi,g,|G|) < z(\Delta, h,k)$ at least when  both $\Delta (G) \geq \frac{8}{g-2}$ and $|G| > 8 + \frac{16}{g-2}$ hold;
 \item[(b)] $s(\chi,3,|G|) < j_1(\Delta, \chi)$ at least when both $\Delta (G) \geq 11$ and $-\frac{|G|^2}{24}\leq \chi$ hold;
\item[(c)] $s(\chi,g,|G|) < j_2(\Delta, \chi, g)$ at least when both
   $\Delta (G) \geq \frac{7}{2} + \frac{6}{g-2}$ and $-\frac{|G|^2}{8}(1 - \frac{2}{g}) \leq \chi$ hold,
 \end{itemize}
where (under the notation of Theorems \ref{edno}, \ref{thm:GZ} and \ref{thm:H21}): 
 $s(\chi,g,|G|) = 3 + \frac{8}{g-2} - \frac{4\chi g}{|G|(g-2)}$, 
$z(\Delta, h,k) = \min\{\Delta(G)+h+2,\Delta(G)+k+1\}$, 
$j_1(\Delta, \chi) = \Delta(G)+\sqrt{12-6\chi}+1/2$ and 
$j_2(\Delta, \chi, g) = \Delta (G) + \frac{\sqrt{8g(2-g)\chi + (3g-2)^2} -g +6}{2(g-2)}$. 
\end{rmk}

\section{Proof of the main result}
The average degree of a graph $G$ is defined as $ad(G) = 2\|G\|/|G|$.  
For the proof of Theorem \ref{edno} we needs the following lemmas. 
\begin{lem}[Hartnell and Rall~\cite{HR1}]\label{HR1}
For any graph $G$, $b(G) \leq 2ad(G) -1$.
\end{lem}
\begin{lem} \label{SamAver}
Let $G$ be a connected graph embeddable on a surface whose Euler
characteristic $\chi$ is as large as possible and let $g(G) = g < \infty$. 
\begin{enumerate}
\item[(i)] Then $ad(G) \leq \frac{2g}{g-2}(1-\frac{\chi}{|G|})$.
\item[(ii)] If $G$ contains no intersecting $g$-faces, then 
$ad(G) \leq 2 + \frac{4g+2}{g^2-g}-2(1+\frac{2}{g-1})\frac{\chi}{|G|}$.
\item[(iii)] If G contains no adjacent $g$-faces, then
$ad(G) \leq \frac{2g(g+1)}{g^2-g-1}(1-\frac{\chi}{|G|})$.
\end{enumerate}
\end{lem}
\begin{proof}
By Euler's formula, $f(G) = \chi - |G| + \frac{1}{2}ad(G)|G|$.

(i) We have $ad(G)|G| = 2\|G\| = \Sigma_{i \geq g} if_i(G) \geq gf(G) = g(\chi - |G| + \frac{1}{2}ad(G)|G|)$ 
and the result easy follows.

(ii) Since $G$ contains no intersecting $g$-faces, each vertex is incident to at most one $g$-face.
This implies $gf_g(G) \leq |G|$. Hence 
$ad(G)|G| = 2\|G\| = \Sigma_{i \geq g} if_i(G) \geq (g+1)f(G) - f_g(G) \geq
 (g+1)(\chi - |G| + \frac{1}{2}ad(G)|G|) - \frac{|G|}{g}$.
 After a short computation, the result follows. 

(iii) Since G contains no adjacent $g$-faces, it follows that $gf_g(G) \leq \|G\| = \frac{1}{2}ad(G)|G|$.  
 Hence $ad(G)|G| %= 2\|G\| = \Sigma_{i \geq g} if_i(G) 
\geq (g+1)f(G) - f_g(G) \geq
 (g+1)(\chi - |G| + \frac{1}{2}ad(G)|G|) - \frac{1}{2g}ad(G)|G|$. 
After some obvious manipulations we obtain the result.
\end{proof}
It remains to note that Theorem \ref{edno} follows by combining 
Lemmas \ref{HR1} and \ref{SamAver}.

\end{document}